\documentclass[11pt, oneside]{article}   	% use "amsart" instead of "article" for AMSLaTeX format
\usepackage{geometry}                		% See geometry.pdf to learn the layout options. There are lots.
\usepackage{graphicx}				% Use pdf, png, jpg, or eps§ with pdflatex; use eps in DVI mode
\usepackage{amssymb}
\usepackage{amsmath}
\usepackage{amsthm}
\usepackage{harpoon}
\usepackage{accents}

\usepackage{tikz}  
\usepackage{float}
\restylefloat{table}
\usepackage{mathrsfs}
\usepackage{verbatim}
\usepackage{enumitem}  

\usetikzlibrary{positioning,chains,fit,shapes,calc}  

\newtheorem{theorem}{Theorem}
\newtheorem{lemma}{Lemma}
\newtheorem{definition}{Definition}

\newtheorem{remark}{Remark}
\newtheorem{corollary}{Corollary}

\newcommand{\ppp}{\[\begin{aligned}}
\newcommand{\ooo}{\end{aligned}\]}

\begin{document}

\title{A Boolean polynomial operator for the Collatz $3n+1$ problem}
\author{Mario DeFranco}
\maketitle

\abstract{We define a reformulation of the Collatz $3n+1$ map which allows us to express it as an operator on the set of sequences of Boolean polynomials. We express this operator in terms of certain carry sequences that arise from addition in base-2, and then we simplify them and find explicit formulas for them.}

\section{Introduction}
The Collatz $3n+1$ problem is a well-known problem in number theory which asks whether successive iterations of the map $C$ (defined below) eventually reach 1 for every positive integer input. See \cite{Lagarias} for more information. This problem is also known under various names such as Hailstone sequence, Hasse's algorithm, Kakutani's problem, the Syracuse problem, the Thwaites conjecture, and Ulam's problem.  

Let 
\begin{align*}
C(k) \colon \mathbb{N} & \rightarrow \mathbb{N} \\ 
k &\mapsto \begin{cases}
\frac{k}{2} &\text{ if $k$ is even} \\ 
3k+1 &\text{ otherwise}.   
\end{cases}
\end{align*}

Let $v_2(k)$ denote the largest integer $r$ such that $2^r$ is a factor of $k$. It is customary to consolidate the divisions by 2 with the map
\begin{align*}
\mathscr{C}(k) \colon \mathbb{N} & \rightarrow \mathbb{N} \\ 
k &\mapsto \frac{3k+1}{2^{v_2(3k+1)}}.
\end{align*}

Instead of adding 1 to $3k$ and then dividing by the power of 2, we instead add the power of 2 to $3k$, leading to the map 
\begin{align*}
\mathscr{O}(k) \colon \mathbb{N} & \rightarrow \mathbb{N}\\ 
k &\mapsto 2^{v_2(k)}+ 3k.
\end{align*}
Using the base-2 expansion of an integer $k$, we express $\mathscr{O}$ as an operator on the set $R^\infty$ of sequences $y$ of elements in a Boolean algebra $R$. 
With 
\[
y = (y(i))_{i=1}^\infty
\]
and 
\[
\mathscr{O}(y) = (\mathscr{O}(y)(i))_{i=1}^\infty,
\]
we prove explicit formulas for $\mathscr{O}(y)(i)$ as a polynomial in the quantities
\[
y(j), 1 \leq j \leq i. 
\]

\section{Defining the Boolean polynomial operator $\mathscr{O}$}

Fix an integer $n$. Let $x_i, 1 \leq i \leq n$ denote indeterminates, and let $I$ denote the ideal 
\[
I = \langle x_i^{2}-x_i \colon 1 \leq i \leq n \rangle. 
\]
Let $R$ denote the ring 
\[
R= \mathbb{F}_2[x_1, \ldots x_n]/ I.
\]
Let $R^\infty$ denote the set of sequences of elements in $R$, and for $y \in R^\infty$ we write  
\[
y = (y(i))_{i=1}^\infty.
\]

\begin{definition} 
Suppose we have two sequences $u=(u(i))_{i=1}^\infty$ and $v=(v(i))_{i=1}^\infty$ in $R^\infty$. We define the  
$\mathrm{Add}(u,v)$ as the sequence
\[
\mathrm{Add}(u,v)= (u(i)+ v(i) + c(u,v;i))_{i=1}^{\infty}
\]
where  $c(u,v;i)$ is defined for $i \geq 1$ by
\begin{align*}
c(u,v;i+1) &= u(i) v(i)+ c(u,v;i)(u(i)+v(i))  
\end{align*}
with $c(u,v;1)=0$. We call $c(u,v) =(c(u,v;i))_{i=1}^{\infty}$ the \emph{carry sequence} for $u$ and $v$. 
\end{definition} 

\begin{definition} 
For $y \in R^\infty$, define the sequence $\mathrm{Shift}_1(y) \in R^\infty$ by 

\[
\mathrm{Shift}_1(y) = (y(i-1))_{i=1}^\infty
\]
with $y(0)$ denoting 0. 

Define $P(y;j)$ to be 
\[
P(y;j) =  \prod_{i=1}^{j}(1+y(i)).
\]

Define the sequence $\mathrm{PlusPowerTwo}(y)$ by 
\[
\mathrm{PlusPowerTwo}(y)= (y(j) P(y;j-1))_{j=1}^\infty.
\]
Define the sequence $\gamma(y) = (\gamma(y; j))_{j=1}^\infty$ to be the carry sequence for $y$ and $\mathrm{PlusPowerTwo}(y)$. 
Define the sequence $\delta(y) = (\delta(y; j))_{j=1}^\infty$ to be the carry sequence for $\mathrm{Add}(y,\mathrm{PlusPowerTwo}(y))$ and $\mathrm{Shift}_1(y) $.  
\end{definition} 

Thus let $\sigma$ be a specialization of the $x_i, 1 \leq i \leq n,$ to values in $\mathbb{F}_2$, and suppose  $y \in R^\infty$ satisfies 
\[
y(i)\big |_\sigma =0
\]
for all but finitely many $i$. Then $y\big|_\sigma$ corresponds to the base-2 expansion of some integer $k$ via 
\[
k = \sum_{i=1}^\infty y(i) \big |_\sigma 2^{i-1}. 
\] 
 The base-2 expansion of the sum $k+2^{v_2(k)}$ is therefore the sequence 
\[
\mathrm{Add}(y,\mathrm{PlusPowerTwo}(y))\big|_\sigma.
\]
The base-2 expansion of $2k$ is 
\[
\mathrm{Shift}_1(y)\big|_\sigma. 
\]
The base-2 expansion of $3k+2^{v_2(k)}$ is then 
\[
\mathrm{Add}(\mathrm{Add}(y,\mathrm{PlusPowerTwo}(y)), \mathrm{Shift}_1(y))\big|_\sigma
\]
which we denote by $\mathscr{O}(y)$.

\begin{theorem}
For $y \in R^\infty$ with 
\[
\mathscr{O}(y) = (\mathscr{O}(y)(i))_{i=1}^\infty,
\]
we have 
\[
\mathscr{O}(y)(i) = \delta(y; i)+ \gamma(y;i) + y(i) P(y; i-1) + y(i) + y(i-1). 
\] 
\end{theorem}
\begin{proof}
This follows immediately from the definitions after applying the two $\mathrm{Add}$ operations.  
\end{proof}

The remainder of this paper consists of finding explicit expressions for $\gamma(y)$ and $\delta(y)$ in terms of the following expressions.

\begin{definition} 

\begin{align*}
U(y, a, b) &= \prod_{i=1}^{b-a+1}y(b-i+1) \\
V_1(y, a, b) &= \prod_{i=1}^{b-a+1}( y(b-i+1)+ \mathbf{1}(i \equiv 0 \pmod 2)) \\
V_2(y,a,b) &= \prod_{i=1}^{b-a+1}( y(a+i-1)+ \mathbf{1}(i \equiv 0 \pmod 2)) \\
P(y, a, b) &= \prod_{i=1}^{b-a+1}( 1+y(b-i+1)) \\
\end{align*}
\end{definition}

\section{Formula for $\gamma$}

\begin{theorem} \label{t gamma sum}
\[
\gamma(y;j+1) = \sum_{k=0}^{\lfloor \frac{j-3}{2} \rfloor } P(y;j-2k-2) U(j-2k,j) %\prod_{i=0}^{2k}y(j-i) 
 \]
\end{theorem}
\begin{proof}
First note that for all $j$, $\gamma(y;j+1)$ is polynomial in $\mathbb{F}_2[y(1), \ldots, y(j)]$ with zero constant term. This follows from straightforward induction on $j$.  
 Since $P(y;j)$ by definition has a factor of $1+y(i)$ for each $1 \leq i \leq j$, we obtain  
 \[
 \gamma(y;j+1) P(y;j)=0 
 \]
 since $y(i)(1+y(i))=0$. Therefore, from the definition of carry sequence, we may express the recursive definition for $\gamma(y;j+1), j>0$ as
\[
\gamma(y;j+1) = y(j)(P(y;j-1)+ \gamma(y;j))
\]
which yields 
\[
\gamma(y;j+1) = \sum_{i=0}^{j-1} P(y;j-i-1) \prod_{k=0}^i y(j-k).
\]
Now the term at index $i$ consists of all expressions of the form $y(S)$ with $[j-i,j] \subseteq S$. Consider $y(T)$ with $[j-i,j] \subseteq T$ but $j-i-1 \notin T$. Then $y(T)$ appears at each index $0 \leq  i' \leq i$. Thus its coefficient in the sum is $i+1$, so we have only those $y(T)$ with even $i$ contributing. This completes the proof.    
\end{proof}

\section{Formula for $\delta$}

\begin{lemma} \label{l gd}
For $y \in R^\infty$ and $j \geq 1$,
\[
\gamma(y;j) \delta(y;j) = 0.
\]
\end{lemma}
\begin{proof} 
We use induction on $j$. The statement is true for $j=1$. Assume it is true for some $j \geq 1$.  Suppose there is some specialization $\sigma$ of the $x_i$ to values in $\mathbb{F}_2$ such that for a fixed $j$ that $\delta(y;j+1)$ and $\gamma(y;j+1)$ are both equal to 1 under $\sigma$. For the remainder of the proof we apply $\sigma$ to all $x_i$. Now by the definition of the carry sequence, $\delta(y;j+1)$ equals 1 exactly when at least two of the three quantities 
\begin{align}
\delta(y;j),\\ 
\gamma(y;j)+ y(j)P(y;j-1)+ y(j),\\ 
y(j-1)
\end{align} 
are equal to 1. 
Likewise $\gamma(y;j+1)$ equals 1 exactly when at least two of the three quantities
\begin{align}
\gamma(y;j),\\ 
y(j)P(y;j-1),\\ 
y(j)
\end{align} 
are equal to 1. 

Case $y(j-1)=1$: Then $P(y;j-1)=0$, so  $y(j) = \gamma(y;j) =1$. By induction hypothesis we must have $\delta(y;j)=0$.  But 
\[
\gamma(y;j)+ y(j)P(y;j-1)+ y(j) = 1+0+1\equiv 0 \pmod 2. 
\]

So we may assume $y(j-1)=0$. 

Case $y(j)=1$ and $y(j-1)=0$: Since $y(j-1)=0$, we must have $\delta(y;j) = 1$. By induction hypothesis, we have $\gamma(y;j) = 0$, and so  $y(j) P(y;j-1) =1$. But then 
\[
\gamma(y;j)+ y(j)P(y;j-1)+ y(j) = 0+1+1\equiv 0 \pmod 2. 
\]

Therefore we see $y(j)=0$ which contradicts the assumption that $\gamma(y;j+1)=1$. Thus $\gamma(y;j+1) \delta(y;j+1) = 0$ for any specialization $\sigma$, so this product is equal to zero as an element of $R$. This completes the induction step and the proof. 
\end{proof}

\begin{definition} 
\begin{align*}
\delta_1(y; j) &= \sum_{k=0}^{j-1} \gamma(y;j-k) \prod_{i=1}^k (y(j-i+1)+y(j-i))\\
\delta_2(yl j) &= \sum_{k=0}^{j-2} y(j-k)y(j-k-1) \prod_{i=1}^k (y(j-i+1)+y(j-i))
\end{align*}
\end{definition}

\begin{lemma} \label{l delta sums}

\begin{align}
\delta(y;j+1) = \delta_1(y;j)+ \delta_2(y;j)
\end{align}
\end{lemma}
\begin{proof}
By the definition of $\delta(y)$ as a carry sequence, we have 
 \[
\delta(y;j+1) =
\delta(y;j)(\gamma(y;j)+ y(j)P(y;j-1)+ y(j)+y(j-1))+ (\gamma(y;j)+ y(j)P(y;j-1)+ y(j))y(j-1).
 \]

 By the same reasoning in proof of Theorem \ref{t gamma sum} we see that 
 \[
 \delta(y;j) \in \mathbb{F}_2[y(1), \ldots, y(j-1)],
 \] so
 \[
 \delta(y;j) P(y;j-1)=0.
 \]
 It also follows from Theorem \ref{t gamma sum} that 
 \[
y(j-1) \gamma(y;j)=\gamma(y;j).
 \]
 Combing these facts with Lemma \ref{l gd} and 
 \[
 y(j-1)P(y;j-1)=0
 \]
allows us to simplify the recursive definition of $\delta(y;j+1)$ as 
 \[
\delta(y;j+1) =
 \gamma(y;j)+ y(j)y(j-1)+( y(j)+y(j-1))\delta(y;j). 
 \]
 This implies the formula in the lemma statement. This completes the proof. 
\end{proof}

\subsection{Formula for $\delta_2$}
\begin{lemma} \label{l sum2}
\begin{align*} \label{sum 2}
y(j-k-1)y(j-k) \prod_{i=0}^k (y(j-k+i+1)+y(j-k+i))   
& = y(j-k-1) V_2(y, j-k,k)
\end{align*}
\end{lemma}
\begin{proof}
We have by definition
\[
y(j-k-1) V_2(y, j-k,k)= y(j-k-1)(\prod_{i=0}^{\lfloor \frac{k+1}{2} \rfloor}y(j-k+2i) \big) \big (\prod_{i=0}^{\lfloor \frac{k}{2}\rfloor}(1+y(j-k+1+2i)) \big).
\]
We claim for $L\geq 0$ that
\begin{equation} 
y(j-k)\prod_{i=0}^L (y(j-k+i+1)+y(j-k+i))  = (\prod_{i=0}^{\lfloor \frac{L+1}{2} \rfloor}y(j-k+2i) \big) \big (\prod_{i=0}^{\lfloor \frac{L}{2}\rfloor}(1+y(j-k+1+2i)) \big).
\end{equation}
We use induction on $L$. The statement is true for $L=0$. Assume it is true for some $L \geq 0$. 
Case $L$ is odd: By induction hypothesis, the right side has a factor of $y(j-k+L+1)$. Multiplying the right side by $y(j-k+L+2)+y(j-k+L+1)$ is thus equivalent to multiplying by 
\[
1+y(j-k+L+2)
\] 
because 
\[
y(j-k+L+1)(y(j-k+L+2)+y(j-k+L+1)) = y(j-k+L+1)(1+y(j-k+L+2)).
\]

Case $L$ is even: By induction hypothesis, the right side has a factor of $1+y(j-k+L+1)$. Multiplying the right side by $y(j-k+L+2)+y(j-k+L+1)$ is thus equivalent to multiplying by 
\[
y(j-k+L+2)
\] 
 because 
 \[
y(j-k+L+1)(1+y(j-k+L+1))=0.
 \]
This completes the induction step. Setting $L=k$ completes the proof. 
\end{proof}

\begin{theorem} \label{t delta2}
\[
\delta_2(y; j+1) = \sum_{i=1}^{j-1} y(i)V_2(i+1,j)
\]
\end{theorem}
\begin{proof}
This follows from Lemma \ref{l sum2} applied to the definition of $\delta_2$.
\end{proof}

\subsection{Formula for $\delta_1$}

\begin{definition} 
Let $S \subseteq [1,t]$ be a set ordered as 
\[
S = ( S(1), S(2), \ldots, S(k).)
\]
with $S(i) < S(i+1)$. We denote the maximal index $k$ by $|S|$. Suppose $y \in R^\infty$. 
Define $y(S)$ to be 
\[
\prod_{i\in S} y(i).
\] 
Define the length $\ell(S)$ of a set $S$ to be 
\[
\ell(S) = \max(S) - \min(S)+1.
\]
We say that an interval $[a,b]$ where $b \geq a$ is a \emph{block} of $S$ if 
\[
[a,b] \subseteq S
\] 
and
\[
a-1 \notin S \text{ and } b+1 \notin S. 
\]

Define the block sequence $B$ of $S$ to be 
\[
B = ( B_1, \ldots, B_l)
\]
where $B_i = [a_i, b_i]$ is a block of $S$, $b_i < a_{i+1}-1$, and 
\[
\bigcup_{i=1}^l B_i = S. 
\]
We denote the maximal index $l$ by $|B|$.

We say that a block $[a,b]$ is an \emph{interior block in $[1,t]$} if $1\neq a$ and $t\neq b$. Otherwise we say that it is an \emph{exterior block in $[1,t]$}. 
\end{definition}

\begin{definition}
Define $G_{j}$ to be the set of subsets $S \subseteq [1,j-1]$ such that 
\[
\gamma(y;j) = \sum_{S \in G_j} y(S).
\] 

\end{definition}

\begin{remark} \label{r G}
From Theorem \ref{t gamma sum}, each $S\in G_j$ has a block sequence whose last block is an exterior block in $[1,j-1]$ of odd length.  
\end{remark}

\begin{definition} 
 Define a \emph{skip} of $B$ in $[1,t]$ to be a block for the set $[1,t] \backslash S$.  

Define $\lambda(S)$ to be either the smallest index $j$ such that 
\[
b_{i}+ 2 = a_{i+1}
\]
for all $i \geq j$, or to be $|B|$ if no such index exists.  

We call $S$ a \emph{skip-1 set in $[a,b]$} if $\min(S)=a$, $\max(S) = b$, and every block of $[a,b] \backslash S$ is of length 1. 

Define $\kappa(S)$ to be either the smallest index in $ [\lambda(S), |B|)$ such that 
\[
\ell(B_i) \equiv 1 \pmod 2, 
\]
for all $|B|>i >\kappa(S)$, or to be $|B|$ if no such index exists. 

Define $\epsilon(S,t)$ to be 
\[
\epsilon(S,t) = \begin{cases} 
1 &\text{ if } t \in S \\ 
0 &\text{ otherwise}
\end{cases}.
\]

Let $\mathrm{lead}(S,t) \subseteq [1,t]$ denote the set 
\[
\mathrm{lead}(S,t) =\bigcup_{i=\kappa(S)}^{|B|- \epsilon(S, t)} B_i.
\]
Thus $\mathrm{lead}(S,t)$ is a skip-1 set in $[a_{\kappa(S)}, b_{|B|-\epsilon(S)}]$. 

Let $\Lambda(t)$ denote the set  
\[
\Lambda(t) = \{ S \subseteq [1,t] \colon \ell(\mathrm{lead}(S,t)) \equiv 1 \text{ or } 2 \pmod 4 \}
\]
and let $\Lambda(t, a,b) \subseteq \Lambda(t)$ denote the set of those $S$ such that 
\[
\min(\mathrm{lead}(S,t)) = a \text{ and } \max(\mathrm{lead}(S,t)) = b.
\]
For $i \in \{ 0,1\}$, let $T_i(a,b)$ denote the set of skip-1 sequences $Q$ in $[a,b]$ such that the length of the first block of $Q$ (i.e. the block that contains $a$) has length equivalent to $i \pmod 2$; such that every other block has odd length; and such that 
\[
\ell(Q) = i-1 \pmod 4. 
\]

%For $i \in \{ 0,1\}$m let  $\Lambda_i(t, a,b) \subseteq \Lambda(t, a,b) $ denote the set those $S$ such that the first block of $\mathrm{lead}(S,t)$ (the block containing $a$) has length equivalent to $i \pmod 2$.   
%Define $w_1(S,t) \in \{ 0,1\}$ to be 
%\[
%w_1(S,t) \equiv 
%\begin{cases} 
%1 &\text{ if } \ell(\mathrm{lead}(S,t)) \equiv 1 \text{ or } 2 \pmod 4 \\ 
%0 &\text{ otherwise}. 
%\end{cases}
%\]
%\[
%w_1(S) \equiv  \sum_{i=\kappa(S)}^{|B| - \epsilon(S)} \big \lceil \frac{\ell(B_i)}{2} \big \rceil \pmod 2. 
%\]
\end{definition}

\begin{lemma} \label{l ceil mod 4}
\[
\lceil \frac{L}{2}\rceil \equiv \begin{cases} 1 \pmod 2 \text{ if } L_1 \equiv 1 \text{ or } 2 \pmod 4\\ 
0 \pmod 2 \text{ otherwise.}
\end{cases}
\]
\end{lemma}
\begin{proof} 
This follows from straightforward checking of the cases $L \mod 4$. This completes the proof.  
\end{proof} 

\begin{lemma} \label{l add ceils}
For any integers $L_i \geq 0$,
\[
\sum_{i=1}^m \lceil \frac{L_i}{2}\rceil \equiv \lceil \frac{m-1+\sum_{i=1}^m L_i}{2}\rceil  \pmod 2. 
\] 
\end{lemma}

\begin{proof} 
We use induction on $m$. The statement is true for $m=1$. Checking cases mod 4 using Lemma \ref{l ceil mod 4} also shows it is true for $m=2$. Assume it is true for some $m \geq 2$. Then 
\begin{align*} 
\sum_{i=1}^{m+1} \lceil \frac{L_i}{2}\rceil &=  \lceil \frac{L_{m+1}}{2}\rceil +\sum_{i=1}^m \lceil \frac{L_i}{2}\rceil \\ 
&\equiv \lceil \frac{L_{m+1}}{2}\rceil +\lceil \frac{m-1+\sum_{i=1}^m L_i}{2}\rceil \pmod 2\\
&\equiv  \lceil \frac{1+L_{m+1} +m-1 \sum_{i=1}^m L_i}{2}\rceil \pmod 2\\
\end{align*}
where on the last line we have used the statement for $m=2$. This completes the induction step and the proof. 
\end{proof}

\begin{lemma} \label{l Ij}
Let $I_j$ denote the set of subsets $S$ of $[1,j]$ such that every interior block $B_i$ of $S$ in $[1,j]$ has odd length and every skip of $S$ in $[1,j]$ has size 1, and such that there is at least one such skip.  
Then
\[
\prod_{k=1}^{j-1} (y(j-k+1)+ y(j-k)) = \sum_{S \in I_j} y(S).
\]
Note that $I_1$ consists of exactly one subset $S = \emptyset$. 
\end{lemma}
\begin{proof}
 In the expansion of the left side of the lemma statement, a monomial $t$ corresponds to a word $w$ of length $j-1$ in letters $L$ and $R$, where the $k$-the letter is $L$ if $y(j-k+1)$ is chosen from $(y(j-k+1)+ y(j-k))$, and $R$ otherwise. We write 
 \[
 w = \prod_{i=1}^m R^{r_i} L^{l_i}
 \]
 where $r_1\geq 0, l_m \geq 0$, and all other $r_i, l_i \geq 1$. A skip occurs exactly when $LR$ is a subword of $w$ or when $r_1=0$ or when $l_m =0$. These skips are all of length 1. The set $\{ R^{r_i}L^{l_i}\colon 1 \leq i \leq m\}$ is in bijection with the set of blocks for the set $S$ with $t = a(S)$. For a fixed $i$, suppose that the interior block $B_i$ corresponding to $ R^{r_i}L^{l_i}$ contributes a factor of 
 \begin{equation} \label{factor contribution}
\prod_{v=a_i}^{b_i} x_{v}
 \end{equation}
 to the monomial $t$. This block has length
 \[
\ell(B_i)= r_i + l_i -1.
 \] 
For this $i$, we may alter 
\begin{align*}
r_i & \mapsto u_i \\ 
l_i & \mapsto \ell(B_i)+1-u_i \\ 
\end{align*}
for $1 \leq u_i \leq \ell(B_i)$ and still preserve the block $B_i$ and thus also preserve the factor contribution \eqref{factor contribution}. Thus the total number of times the monomial $t$ appears in the expansion is 
\[
\prod_{B_i \text{ is an interior block for $S$ in $[1,j]$}} \ell(B_i).
\]
Thus the coefficient of $t$ is non-zero if and only if all such $\ell(B_i)$ are odd (or if there are no interior blocks). This completes the proof.  
\end{proof}

\begin{lemma}  \label{l 1Lambda}
\[
\sum_{k=0}^{j-1} \gamma(y;j-k) \prod_{i=1}^k (y(j-i+1)+y(j-i))  = \sum_{S \in \Lambda(j)} y(S).  
\]
\end{lemma} 
\begin{proof}
By Lemma \ref{l Ij} and Remark \ref{r G},
 \[
 \gamma(y;j-k) \prod_{i=1}^k (y(j-i+1)+y(j-i)) = \sum_{\substack{S_1 \in G_{j-k}\\ S_2 \in I_{[j-k+1,j]}}} y(S_1)y(S_2).
 \]
 Consider $S \subseteq[1,j]$ with block sequence $(B_1, \ldots, B_l)$. We count the number of $k$ for which we have 
 \[
 S = S_1 \cup S_2
 \] 
 with $S_1 \in G_{j-k}$ and $S_2 \in I_{[j-k+1,j]}$. For such a pair $(S_2, S_2)$, let $M$ denote the greatest element of $S_1$. Then $M$ must appear in a block $B_i$ of $S$ with $i \geq \lambda(S)$ (for otherwise $S_2$ would then have a skip of length greater than 1, contrary to its definition). Also we must have $ |B| \geq i \geq \kappa(S)$ (for otherwise $S_2$ would have an interior block of even length, also contrary to its definition). Note that we can have $i=|B|$ only in the case that the last block of $B$ does not contain $j$, for otherwise  $S_2$ would then have no skip. (This is the meaning of $\epsilon(S)$.) This block sequence is the definition of $\mathrm{lead}(S,j)$:
 \[
  \bigcup_{i=\kappa(S}^{|B| - \epsilon(S)}B_i= \mathrm{lead}(S,j)
 \]
  
 Thus suppose $M \in B_i = [a_i, b_i]$. Thus the last block of $S_1$ is of the form $[a_i, M]$, and since it has odd length by definition, the possible choices for $M$ are $a_i, a_i+2, a_i+4, \ldots, c_i$ with 
 \[
 c_i = \begin{cases} 
 b_i &\text{ if } b_i \equiv a_i \pmod 2 \\ 
 b_i -1 &\text{ otherwise}. 
 \end{cases} 
 \]    
 The number of such choices is thus
 \[
 \lceil \frac{\ell(B_i)}{2} \rceil.
 \]
 Adding these up for all the blocks in $\mathrm{lead}(S,j)$ gives
 \begin{align*}
\sum_{i=\kappa(S)}^{|B| - \epsilon(S)}  \lceil \frac{\ell(B_i)}{2} \rceil  &\equiv  \lceil  \frac{|B| - \epsilon(S)-\kappa(S)+\sum_{i=\kappa(S}^{|B| - \epsilon(S)} \ell(B_i)}{2}\rceil \pmod 2\\
&\equiv   \lceil \frac{\ell(\mathrm{lead}(S,j)}{2} \rceil \pmod 2. 
 \end{align*}
 where we have used Lemma \ref{l add ceils}. This is non-zero if and only if 
 \[
 \ell(\mathrm{lead}(S,j)) \equiv 1\text{ or } 2 \pmod 4.
 \]
 Thus the sets $S$ that have a non-zero contribution are exactly those in $\Lambda(j)$. This completes the proof. 
\end{proof}

\begin{lemma} \label{l V}
Recall that $T_1(a,b)$ denotes the set of skip-1 sets $S$ in $[a,b]$ such that every block of $S$ has odd length.   Then 
\[
\sum_{S \in T_1(a,b)} y(S) =  V_1(y,a,b). 
\]
\end{lemma}
\begin{proof} 
Note that the assumption that every block of a skip-1 set in $[a,b]$ has odd length implies that $a \equiv b \mod 2$.   

%\[
%\#\{ i \colon \ell(B_i) \equiv 1 \pmod 4 \} \equiv \frac{b-a}{2}+1 \pmod 2.
%\]

We use induction on $b-a$. The statement is true when $a=b$ for then $T_1(a,a) = (a)$ and $V_1(x,a,a) = x(a)$. Assume it is true for all $a,b$ with $a-b \leq 2n$ for some $n \geq 0$. Assume $b=a+2n+2$. Then we partition $T(a,b)$ into subsets $T_1^{(1)}(a,b)$ and $T_1^{>1}(a,b)$ where $T_1^{(1)}(a,b)$ consists of those sets $S$ whose block sequence contains the block $[b,b]$, and $T_1^{>1}(a,b)$ consists of the remaining sets. Then 
\begin{align*}
\sum_{S \in T_1^{(1)}(a,b)} y(S)  &= y(b)\sum_{S \in T_1^{(1)}(a,b-2)} y(S) \\ 
&= y(b) V_1(x,a,b-2)
\end{align*}
by the induction hypothesis. And 
\begin{align*}
\sum_{S \in T_1^{>1}(a,b)} y(S)  &= y(b)y(b-1)\sum_{S \in T_1(a,b-2)} y(S) \\ 
&= y(b) y(b-1)V(y,a,b-2).
\end{align*}
Adding the two results gives 
\[
y(b)(1+y(b-1))V_1(y,a,b-2) = V_1(y,a,b). 
\]
This completes the induction step and the proof. 
\end{proof}

\begin{theorem} \label{t 1 2 sum}
\begin{align}
\sum_{S \in \Lambda(t)} y(S)=& \sum_{j=0}^{\lfloor \frac{t-3}{4}\rfloor}\sum_{i =-1}^{t-4j-4}P(y,1,i) V_1(y,i+2,i+3+4j) U(y,i+5+4j,t) \label{l 2 sum}\\
& + \sum_{i=0}^{\lfloor \frac{t-2}{4}\rfloor} V_1(y,1,4i+1) U(4i+3, t) \label{l 1 sum}
\end{align}
where we denote $P(1,-1) = P(1,0) = U(t+1,t) = 1$. 
\end{theorem}
\begin{proof}

Suppose $1 \leq a < b <t$. 

%let $H(t,r,s)$ denote the set of subsets
%\[
%H(t,r,s) = \{ [s,t] \cup M \colon  M \subseteq [1,r]\}.
%\]
%denoting $[r,t]  = \emptyset$ for $r > t$. For $Z \in H(t,a,b)$,  let $\Lambda_i(t,a,b, Z) \subseteq \Lambda_i(t,a,b)$ denote the set of those $S$ such %that 
%\[
 %\Lambda_i(t,a,b, Z)  = \{S \in \Lambda_i(t,a,b) \colon  S \backslash [a,b] = Z \}.
%\]
%Let $\bar{\Lambda}_i(t,a,b)$ denote the set
%\[
%\bar{\Lambda}_i(t,a,b)=\{\mathrm{lead}(S,t) \colon S \in \Lambda_i(t,a,b) \}.
%\]
By construction we have
\begin{align*}
\Lambda_0(t,a,b) &= \emptyset \text{ if } b-a+1 \not\equiv 2 \pmod 4 \\ 
 \Lambda_1(t,a,b) &= \emptyset \text{ if } b-a+1 \not\equiv 1 \pmod 4. 
\end{align*}
Thus 
\begin{align}
\sum_{S \in \Lambda(t)} y(S) =& \sum_{a=2}^{t-1} \sum_{i=0}^{\lfloor \frac{t-1-a}{4} \rfloor}((\sum_{S \in \Lambda_1(t,a,a+4i)} y(S))+ (\sum_{S \in \Lambda_2(t,a-1,a+4i)} y(S)))\label{2 sum}\\ 
&+ \sum_{i=0}^{\lfloor \frac{t-2}{4} \rfloor}\sum_{S \in \Lambda_1(t,1,1+4i)} y(S). \label{1 sum}
\end{align}

Assume $b-a \equiv 0 \pmod 4$. Suppose $S \in \Lambda_1(t,a,b)$. Then $S$ is of the form 
\[
S = Z \cup \mathrm{lead}(S,t) \cup [b+2,t]
\]
where $Z$ may be any set 
\[
Z \subseteq [1,a-3].
\]
This follows from the construction of lead sequences: if $W \in T_1(a,b)$ and \[W\cup Z \cup [b+2,t] \in \Lambda_1(t,a,b),\] then we must have $a-1 \notin Z$ and $a-2 \notin Z$. 
Thus
\begin{equation} \label{SZ1}
\sum_{S \in \Lambda_1(t,a,b)} y(S) = \left(\sum_{Z \subseteq [1,a-3]} y(Z) \right )\left (\sum_{W \in T_1(a,b)} y(W) \right ) U(b+2, t) .
%&= \sum_{Z \in H(t,a,b)} P(y; 1,a-2) U(y; b+2,t)\sum_{W \in \bar{\Lambda}_1(t,a,b)} y(W) \\ 
%&= 
\end{equation}

Now
\[
\sum_{W \in T_1(a,b)} y(W) = V_1(y; a,b))
\]
by Lemma \ref{l V}. 
Thus \eqref{SZ1} is equal to 
\begin{equation} \label{PVU}
P(y; 1,a-3)V_1(y; a,b)U(y; b+2,t).
\end{equation}
This shows line \eqref{1 sum} evaluates to line \eqref{l 1 sum}. 

Next, for $a \geq 2$, suppose $S \in \Lambda_0(t,a-1,b)$. Then $S$ is of the form 
\[
S =  Z \cup \mathrm{lead}(S,t) \cup [b+2,t]
\]
where $Z$ may be any set 
\[
Z \subseteq [1,a-3].
\]

\begin{align} 
\sum_{S \in \Lambda_0(t,a-1,b)} y(S) =& \left (\sum_{Z \subseteq [1,a-3]} y(Z) \right ) \left (\sum_{W \in T_0(a-1,b)} y(W) \right ) U(y; b+2, t) \nonumber  \\
=& P(y; 1,a-3)y(a-1)\left(\sum_{W \in T_1(a,b)} y(W) \right)U(y; b+2, t)  \nonumber \\ 
= & P(y; 1,a-3)y(a-1)V_1(y; a,b)U(y; b+2,t).\label{SZ0}
\end{align}

Adding results \eqref{PVU} and \eqref{SZ0} yields 
\[
P(y; 1,a-3)V_1(y; a-1,b)U(y; b+2,t).
\]
This shows that line \eqref{2 sum} evaluates to line \eqref{l 2 sum}. This completes the proof. 
\end{proof}

\begin{corollary} 
 \begin{align}
\delta_1(y;j+1) =& \sum_{j=0}^{\lfloor \frac{t-3}{4}\rfloor}\sum_{i =-1}^{t-4j-4}P(y,1,i) V_1(y,i+2,i+3+4j) U(y,i+5+4j,t) \\
& + \sum_{i=0}^{\lfloor \frac{t-2}{4}\rfloor} V_1(y,1,4i+1) U(4i+3, t) 
\end{align}
\end{corollary}
\begin{proof} 
This follows from Theorem \ref{t 1 2 sum} and Lemma \ref{l 1Lambda}.  
\end{proof}. 

\section{Further Work}

\begin{itemize} 
\item Analyze the carry sequence for 
\[
\mathrm{PlusPowerTwo}(y) \text{ and } \mathrm{Shift_1}(y)
\]
 and the carry sequence for
 \[
\mathrm{Add}(\mathrm{PlusPowerTwo}(y), \mathrm{Shift_1}(y)) \text{ and } y. 
 \]
\end{itemize}

\end{document}